\documentclass[10pt]{amsart} 

\usepackage{amssymb}
\usepackage{amsmath}
\usepackage{amsthm}
\usepackage{enumerate}
\usepackage[numbers]{natbib}
\usepackage{tikz-cd}
\usepackage[english]{babel}
\usepackage{csquotes}
\usepackage{verbatim}
\usepackage{multicol}
\usepackage{lmodern}
\usepackage[T1]{fontenc}
\usepackage{xr-hyper}
\usepackage[hyperindex,breaklinks]{hyperref}

\newtheorem{theorem}{Theorem}[section]

\newtheorem{lemma}[theorem]{Lemma}
\theoremstyle{definition}
\newtheorem{definition}[theorem]{Definition}
\newtheorem{example}[theorem]{Example}

\newtheorem{cor}[theorem]{Corollary}
\newtheorem{remark}[theorem]{Remark}

\numberwithin{equation}{subsection}

\def\lim{\mathop{\mathrm{lim}}\nolimits}
\def\colim{\mathop{\mathrm{colim}}\nolimits}

\def\Hom{\mathop{\mathrm{Hom}}\nolimits}

\def\Pic{\mathop{mathrm{Pic}}}

\def\gl{\mathop{\mathrm{Gl}}\nolimits}

\def\Pic{\mathop{\mathrm{Pic}}\nolimits}

\def\deg{\mathop{\mathrm{deg}}\nolimits}

\def\det{\mathop{\mathrm{det}}\nolimits}

\def\dim{\mathop{\mathrm{dim}}\nolimits}
\def\rk{\mathop{\mathrm{rk}}\nolimits}

\def\colim{\mathop{\mathrm{colim}}\nolimits}
\def\lim{\mathop{\mathrm{lim}}\nolimits}
\def\et{\mathop{\text{\'et}}\nolimits}

\newtheorem{thmx}{Theorem}
\newtheorem{corx}[thmx]{Corollary}

\makeatletter 
\def\l@subsection{\@tocline{2}{0pt}{1pc}{5pc}{}} \def\l@subsection{\@tocline{2}{0pt}{2pc}{6pc}{}} 
\makeatother

\begin{document}
\title[A functorial approach to stability]{A functorial approach to the stability of vector bundles}

\author{Dario Wei\ss mann} 
\thanks{\emph{Email}: dario.weissmann@posteo.de \\
\hspace*{1.4em}\emph{ORCID}: 0000-0002-4005-2847\\ \hspace*{1.4em}Universit\"{a}t Duisburg-Essen, Universit\"{a}tsstra\ss e 2, 45141, Essen, Germany}

\address{
Fakult\"{a}t f\"{u}r Mathematik,
Universit\"{a}t Duisburg-Essen,
Universit\"{a}tsstra\ss e 2,
45141 Essen,
Germany}

\email{\href{mailto: dario.weissmann@posteo.de}{dario.weissmann@posteo.de}}

\date{May 31, 2023}

\begin{abstract}
    On a normal projective variety the locus of $\mu$-stable bundles that
    remain $\mu$-stable on {\it all}
    Galois covers prime to the characteristic 
    is open in the moduli space of Gieseker semistable sheaves.
    On a smooth projective curve of genus at least $2$ this locus
    is big in the moduli space of stable bundles.
    As an application we obtain a very different behaviour of the \'etale fundamental group in positive vs. characteristic $0$.
\end{abstract}

\maketitle

\subsection*{Keywords} Stability $\cdot$ Slope stability $\cdot$ \'Etale fundamental group $\cdot$ \'Etale cover $\cdot$ Moduli of vector bundles
$\cdot$ Functoriality
\subsection*{Mathematics Subjects Classification}
14H60 14D20 14H30 

\section{Introduction}

Consider the stack of vector bundles
on a smooth projective curve $C$ over an algebraically closed field $k$ of
characteristic $p\geq 0$.
Semistability is a property of vector bundles
which is tailored to obtain a moduli space.
Via the Harder-Narasimhan-filtration (HN-filtration for short) it also reveals additional 
structure of the category of vector bundles
and immediately implies that semistability is functorial under pullback by finite separable morphisms.
Even more structure is revealed via the Jordan-Hölder-filtration (JH-filtration for short). However,
in contrast to the HN-filtration the JH-filtration is not unique and thus functoriality fails for stability.

Recently, those morphisms that preserve the stability of vector bundles under pullback have been identified:
for curves these are exactly the genuinely ramified morphisms, see \cite[Theorem 5.3]{bp}.
In higher dimension, genuinely ramified morphisms also preserve stability, see \cite[Theorem 1.2]{bdp}.

The main goal of this paper is to address a way to measure the failure of stability to be functorial under 
all finite separable pullbacks. 
As an application we obtain a very different
behaviour of the \'etale fundamental groups in positive versus characteristic $0$.

Representations of $\pi_{\et}(C)$ correspond to vector bundles of degree $0$
which are trivialized on some \'etale cover over of $C$, see \cite[1.2 Proposition]{ls}.
In positive characteristic these \'etale trivializable bundles are dense in the moduli space 
$M^{ss,r,0}_C$ of semistable bundles of rank $r$ and degree $0$, 
see \cite[Corollary 5.1]{dm}.
This no longer holds in characteristic $0$ as we show that the general bundle remains stable on all \'etale covers (avoiding the characteristic).
Put another way, the \'etale fundamental group has enough information to recover the moduli space in positive characteristic but not in characteristic $0$.

To make our results precise we need a definition.
Call a vector bundle on $C$ {\it prime to $p$ stable} if it remains stable
after pullback by all finite Galois covers $D\to C$ which have degree prime to $p$, see also Definition
\ref{definition-fun-stab}.
The locus of prime to $p$ stable bundles is open - a direct consequence 
of the following theorem.

\begin{thmx}[Theorem \ref{theorem-very-large-cover} for curves]
Let $r\geq 2$ and $C$ be a smooth projective curve over an algebraically closed
field of characteristic $p\geq 0$.
Then there exists a connected \'etale 
prime to $p$ Galois cover $\pi:C_{r-good}\to C$
such that a vector bundle $V$ of rank $r$ is prime to $p$ stable
iff $\pi^{\ast}V$ is stable.
\end{thmx}
An analogous statement holds for $\mu$-stable bundles on a normal projective variety, see Theorem \ref{theorem-very-large-cover}.
Having identified this locus as open one should also address non-emptiness.
Recall that an open subset $U$ of a variety $X$ is called \emph{big}
if $X\setminus U$ has codimension at least $2$ in $X$.

\begin{thmx}[Theorem \ref{theorem-non-empty}]
Let $r\geq 2$.
If $C$ has genus $g_C\geq 2$, then the prime to $p$ stable locus
$M^{p'-s,r,d}_C$ is big in the moduli space of stable bundles
$M^{s,r,d}_C$. More precisely, we have
    \[
        \dim(M^{s,r,d}_C\setminus M^{p'-s,r,d}_C)\leq rr_0(g_C-1)+1,
	 \]
where $r_0$ denotes the largest proper divisor of $r$.
If $p$ is not the smallest proper divisor of $r$, then equality holds.
\end{thmx}

By considering $d=0$ in Theorem $2$, we obtain the different behaviour of the \'etale fundamental group, i.e., the non-density of the \'etale trivializable bundles in characteristic $0$
versus their density in positive characteristic.
\begin{corx}
    Let $C$ be a smooth projective curve of genus $g_C\geq 2$.
    Let $r\geq 2$. Then the stable bundles of rank $r$
    that are trivialized on a prime to $p$ \'etale cover
    are not dense in $M^{s,r,0}_C$.
\end{corx}

In rank $2$ and characteristic $0$ 
such a non-density result has been independently obtained by Ghiasabadi and Reppen,
see \cite[Corollary 4.16]{gr}.

We also note that the density of the \'etale trivializable bundles in positive
characteristic means that we can not extend Theorem 1 nor Theorem 2 to include all covers;
allowing only for covers of degree prime to the characteristic is crucial.

The key observation in proving Theorem 1 is that while stability is in general 
not preserved under pullback by a Galois cover $D\to C$ 
polystability is. In fact, we can say more: 
A stable vector bundle $V$ on $C$ decomposes on $D$
into a direct sum $\bigoplus_{i=1}^n W_i^{\oplus e}$ of pairwise non-isomorphic 
stable bundles $W_i$ all appearing with the same multiplicity $e$.
Furthermore, the Galois
group of $D/C$ acts transitively on the isomorphism classes of the $W_i$, see Lemma \ref{lemma-pullback-galois}.

The construction of the cover $C_{r-good}$ checking for prime to $p$ stability is then split into two parts: 
A cover $C_{r-large}$ checking for the decomposition behaviour if $n\geq 2$ and a cover $C_{r-good}$ including $n=1$.

The cover $C_{r-large}$ is easily constructed using the transitive action of the Galois group.
To include the case $n=1$ the difficulty arises that while all the conjugates of $W=W_1$ by the Galois group are isomorphic
these isomorphisms might not be compatible. We provide a workaround for descending simple invariant bundles.

Pretending that $W$ descends for now allows for a comparison of the linearizations of $V$ on $D$ and $W^{\oplus e}$.
This gives rise to a $\gl_e$ representation of the Galois group. Finite subgroups prime to the characteristic of $\gl_e$ are 
well-understood. By Jordan's theorem - which in positive characteristic is due Brauer and Feit - 
they are close to being abelian. This allows us to find a cover which also checks for this decomposition behaviour.

The same type of cover works in higher dimensions. However, the workaround for descend only works for curves.
To obtain Theorem 1 in higher dimensions we carefully set up the requirements 
for the workaround of descend and then use a restriction theorem for stability to reduce to dimension $1$.

Theorem 2 is obtained by a dimension estimate on the strata defined by the decomposition behaviour of a stable bundle.

The paper is structured as follows.
In §2 we define functorial notions of stability and study them for genus $g_C\leq 1$. 
We also collect some preliminary properties of (semi)stable bundles under pullback 
as well as a descend lemma for (not necessarily \'etale) flat Galois covers of normal varieties.

In §3 we prove the key lemma. Then we construct the prime to $p$ cover $C_{r-good}$ that checks whether a vector bundle is prime to $p$ stable.

In §4 we investigate certain strata which arise as the complement of the prime to $p$ stable locus 
and estimate their dimension. We work with arbitrary \'etale Galois covers
and obtain Theorem $2$ by considering the cover constructed in Theorem $1$.
We also provide a descend lemma for \'etale cyclic covers 
which may be of independent interest, see Lemma \ref{lemma-cyclic-descend}.

\subsection*{Notation}
    We work over an algebraically closed field $k$ of characteristic $p\geq 0$.
    A \emph{variety} is a separated integral scheme of finite type over $k$.
    A \emph{curve} is a variety of dimension $1$.
    The \emph{function field} of a variety $X$ is denoted by $\kappa(X)$.
 
    If $X$ is a projective variety, then we implicitly choose an ample bundle
    $\mathcal{O}_X(1)$ on $X$. If we consider a finite morphism $\pi:Y\to X$ we
    set $\mathcal{O}_Y(1)=\pi^{\ast}\mathcal{O}_X(1)$.
    By \emph{(semi)stability} we mean $\mu$-(semi)stability
    of reflexive sheaves with respect to $\mathcal{O}_X(1)$.
	
    We denote the \emph{moduli space} of (semi)stable vector bundles of 
    rank $r$ and degree $d$ on a smooth projective 
    curve $C$ by $M^{s,r,d}_C$ (resp. $M^{ss,r,d}_C$).

    Given a morphism $\pi:Y\to X$ of varieties and a sheaf $F$ on $X$ 
    we denote the pullback $\pi^{\ast}F$ also by $F_{\mid Y}$.
 
    By a \emph{cover} $Y\to X$ of varieties we mean a finite separable morphism of 
    varieties, i.e., a finite dominant morphism such that the extension of function 
    fields $\kappa(Y)/\kappa(X)$ is separable.
    A cover is called \emph{Galois} if the extension of function fields 
    $\kappa(Y)/\kappa(X)$ is Galois.
    An \'etale (Galois) cover is a (Galois) cover $Y\to X$ which is \'etale.
	
\section{First observations}
We start by collecting some elementary results on pullback and semistability as well 
as descent theory for flat Galois covers, which is slightly trickier than for 
\'etale Galois covers.
Then we introduce the functorial notions of stability 
and give a complete analysis for smooth projective curves of genus $\leq 1$.

\subsection{Preliminaries on Stability and Pullback}
In this subsection we recall several notions of stability
as well as the basic properties of $\mu$-(semi)stable vector bundles under
pullback. We also include a descent lemma along (possibly non-\'etale) flat 
Galois covers in terms of linearizations.

We begin by recalling semistability, the reader is referred to \cite[Chapter 1, 4]{hl} for a detailed account. On a smooth projective curve $C$ we have two numerical invariants attached to a vector bundle $V$: the \emph{rank} $\rk(V)$
and the \emph{degree} $\deg(V)$. 
This allows us to define the \emph{slope} $\mu(V):=\deg(V)/\rk(V)$ which in turn is used to define (semi)stability. The vector bundle $V$ is called \emph{semistable}
if for all subbundles $0\neq W\subsetneq V$ we have $\mu(W)\leq \mu(V)$.
It is called \emph{stable} if the inequality is strict for all subbundles $0\neq W\subsetneq V$.

These notions are tailored to obtain a moduli space of semistable vector bundles of rank $r$ and degree $d$ which we denote by $M^{ss,r,d}_C$. The closed points of $M^{ss,r,d}_C$ correspond to the \emph{polystable} vector bundles of rank $r$ and degree $d$, i.e., vector bundles
which are a direct sum of stable bundles of the same slope $d/r$.
The moduli space of stable bundles $M^{s,r,d}_C$ is an open subset of $M^{ss,r,d}_C$.

On a normal projective variety $X$ of dimension $\geq 2$
there are several analogues to (semi)stability on a curve.
On the one hand, we have more numerical invariants attached to a coherent sheaf $\mathcal{F}$: (the coefficients of) the Hilbert polynomial 
\[
P(\mathcal{F})(n)=\sum_{i=0}^{\dim(X)}\frac{\alpha_i(\mathcal{F})}{i!}n^i.
\]
On the other hand, the Hilbert polynomial depends on the choice of a polarization $\mathcal{O}_X(1)$ of $X$. We implicitly fix the polarization - also see the notations.

A torsion-free coherent sheaf $\mathcal{F}$ on $X$ is called \emph{Gieseker-semistable} if for all saturated subsheaves $0\neq \mathcal{G}\subsetneq \mathcal{F}$
we have $p(\mathcal{G})\leq p(\mathcal{F})$, where $p(\mathcal{F}):=P(\mathcal{F})/\rk(F)$ is the \emph{reduced Hilbert polynomial}.
The ordering is via the lexicographic ordering on the coefficients of the polynomials starting in the highest degree.
The torsion-free coherent sheaf $\mathcal{F}$ is called \emph{Gieseker-stable}
if the above inequality is strict. As in the curve case these notions lend themselves to a construction of a moduli space of Gieseker semistable torsion-free sheaves.

In this paper we are mostly concerned with the the notion of \emph{$\mu$-stability} which we also abbreviate to \emph{stability}: the \emph{slope} of a coherent sheaf $\mathcal{F}$ which is torsion-free on a big open subset
is defined as 
\[
    \mu(\mathcal{F}):=\deg(\mathcal{F})/\rk(\mathcal{F}),
\]
where the \emph{degree} is defined as
\[
    \deg(\mathcal{F}):=\alpha_{\dim(X)-1}(\mathcal{F})-\rk(\mathcal{F})\alpha_{\dim(X)-1}(\mathcal{O}_X).
\]
We call a \emph{reflexive} sheaf $\mathcal{F}$ \emph{semistable} if for
all saturated subsheaves $0\neq \mathcal{G}\subsetneq \mathcal{F}$ of smaller rank we have $\mu(\mathcal{G})\leq\mu(\mathcal{F})$. A reflexive sheaf $\mathcal{F}$ is \emph{stable} if the above inequality is strict. Further, $\mathcal{F}$ is \emph{polystable} if it is a direct sum of stable sheaves of the same slope $\mu(\mathcal{F})$.

We note that the degree of $\mathcal{F}$ only depends
on its isomorphism class on some big open subset of $X$. 
In particular, we have
$\mu(\mathcal{F})=\mu(\mathcal{F}^{\lor\lor})$, where $\mathcal{F}^{\lor\lor}$
denotes the \emph{reflexive hull} of $\mathcal{F}$.

A cover $Y\to X$ of normal
varieties is flat on a big open subset. As the slope only depends on the isomorphism
class on a big open subset, this is the right setting to study pullback. The basic results are as follows:
\begin{lemma}
\label{lemma-stability-pullback}
	Let $\pi:Y\to X$ be a cover of normal projective
	varieties of degree $d$.
	Let $\mathcal{F}$ be a reflexive sheaf on $X$ and 
        $\mathcal{G}$ be a torsion free sheaf on
	$Y$. Then the following hold:
	\begin{enumerate}[(i)]
		\item $\mu(\mathcal{G})=d(\mu(\pi_{\ast}\mathcal{G})-\mu(\pi_{\ast}\mathcal{O}_Y))$.
		\item $\mu((\mathcal{F}_{\mid Y})^{\lor \lor})=d\mu(\mathcal{F})$.
		\item $\mathcal{F}$ is semistable iff 
            $(\mathcal{F}_{\mid Y})^{\lor \lor}$ 
            is semistable.  
		\item If $\mathcal{F}$ is polystable and $Y\to X$ is Galois, then
			$(\mathcal{F}_{\mid Y})^{\lor \lor}$ is polystable. 
                 If $\pi$ is prime to $p$,
			then $\mathcal{F}$ is polystable iff 
                $(\mathcal{F}_{\mid Y})^{\lor\lor}$ is polystable.
		\item If $(\mathcal{F}_{\mid Y})^{\lor \lor}$ is stable, then so is i              $\mathcal{F}$.
	\end{enumerate}
\end{lemma}

\begin{proof}
    (i) - (iv) are proven in \cite[Lemma 3.2.1 - 3.2.3]{hl}.
    Note that the proofs are independent of the characteristic except for            
    \cite[Lemma 3.2.3]{hl}.
    Here the additional prime to $p$ assumption saves the splitting of the trace.
	
    These results use descent for Galois covers which is a bit trickier than
    for \'etale ones. We spell this out
    in Lemma \ref{lemma-invariant-subsheaf} for flat Galois covers. While
    a Galois cover may be non-flat in general the flat locus is a big open subset.
    The slope only depends on the isomorphism class on an big open subset and Lemma \ref{lemma-invariant-subsheaf} can then be applied to the destabilizing subsheaf 
    as well as the socle after restricting to the flat locus.

    (v): A proper subsheaf of $\mathcal{F}$
    of slope $\geq \mu(\mathcal{F})$ pulls back to a proper subsheaf  
    of $\mathcal{F}_{\mid Y}$ on a big open subset of $Y$ of slope $\geq \mu(\mathcal{F}_{\mid Y})$ by (ii). The claim follows.
\end{proof}

We recall the notions of $G$-invariance and $G$-linearization and prove a descend lemma under flat Galois covers for the latter.
\begin{definition}
    Let $Y\to X$ be a Galois cover of normal varieties with Galois group $G$.
    Thinking of $Y$ as the normal closure of $X$ in $\kappa(Y)$ we obtain an action of $G$ on $Y/X$.
    
    A \emph{$G$-invariant} torsion-free sheaf $V$ on $Y$ is 
    a torsion-free sheaf $V$ together with isomorphisms $\psi_{\sigma}:V\xrightarrow{\sim} \sigma^{\ast}V$ for all $\sigma\in G$.
    By a slight abuse of notation we suppress the choice of the isomorphisms
    and call $V$ a $G$-invariant torsion-free sheaf.
    
    A subsheaf $W\subseteq V$ of a $G$-invariant torsion-free sheaf $V$ is
    called {\it $G$-invariant} if the isomorphisms
    $\psi_{\sigma}:V\xrightarrow{\sim} \sigma^{\ast}V$ induce isomorphisms $W\xrightarrow{\sim}\sigma^{\ast}W$ of subsheaves.

    A torsion-free sheaf $V$ on $Y$ is said to admit a {\it $G$-linearization}
    if for all $\sigma \in G$ there exists an isomorphism $\psi_{\sigma}:V\xrightarrow{\sim} \sigma^{\ast}V$ 
    such that $\tau^{\ast}\psi_{\sigma}\circ \psi_{\tau}=\psi_{\sigma\tau}$ for all $\sigma,\tau\in G$.
\end{definition}

\begin{remark}
    By definition a $G$-invariant subsheaf $W\subseteq V$ of a torsion-free sheaf
    admitting a $G$-linearization
    admits a $G$-linearization as well.
\end{remark}

For an \'etale Galois cover a linearization is the same as a descent-datum. 
This is in general not true for Galois covers or even flat Galois covers, see Example \ref{example-non-descend}.
There is however a version for an invariant saturated subsheaf of a torsion-free sheaves which descends:

\begin{lemma}
\label{lemma-invariant-subsheaf}
    Let $Y\to X$ be a flat Galois cover of normal varieties with Galois group $G$.
    Let $V$ be a torsion-free sheaf on $X$.
    Then a $G$-invariant saturated subsheaf of $V_{\mid Y}$
    descends to a saturated subsheaf of $V$.
\end{lemma}

\begin{proof}
    Let $\eta_Y$ be the generic point of $Y$ and $\eta_X$ the generic point of $X$.
    
    Consider a $G$-invariant saturated subsheaf $W\subseteq V_{\mid Y}$.
    Restricting the inclusion to $\eta_Y$ we obtain a $G$-invariant subvector space 
    $W_{\mid \eta_Y}\subseteq (V_{\eta_X})_{\mid \eta_Y}$.
    
    The field extension $\kappa(Y)/\kappa(X)$ is a $G$-torsor and we can apply descent theory.
     We obtain $W'_{\eta_X}\subseteq V_{\eta_X}$ such that $W'_{\eta_X}\otimes_{\kappa(X)}\kappa(Y)=W_{\eta_Y}$ as subspaces of $(V_{\eta_X})_{\mid \eta_Y}$.
     
     By \cite[Proposition 1]{langton}, 
     which also holds for varieties not just smooth projective varieties, 
     there is a unique saturated subsheaf $W'\subseteq V$ inducing the inclusion
     $W'_{\eta_X}\subseteq V_{\eta_X}$. Pulling back along the flat morphism $Y\to X$ 
     we obtain a saturated subsheaf $W'_{\mid Y}\subseteq V_{\mid Y}$
     which agrees with the inclusion $W_{\eta_Y}\subseteq (V_{\eta_X})_{\mid\eta_Y}$ on the generic point.
     By another application of \cite[Proposition 1]{langton}
     we conclude $W'_{\mid Y}=W$.
\end{proof}

    We provide examples which show that neither "saturated" nor "subsheaf of a sheaf which descends"
    can be removed in Lemma \ref{lemma-invariant-subsheaf}.
\begin{example}
\label{example-non-descend}
    Let $E$ be an elliptic curve and $\pi:E\to \mathbb{P}^1$ be a $2:1$ Galois cover ramified at $4$ points.
    Denote the non-trivial element of the Galois group $G=\mathbb{Z}/2$ by $\sigma$.
    
    Consider a line bundle $L$ of degree $1$ on $E$.
    Then $L\oplus \sigma^{\ast}L$ admits a $G$-linearization, but does not descend to $\mathbb{P}^1$.
    Indeed, if there was a vector bundle $V$ on $\mathbb{P}^1$ such that $V_{\mid E}\cong L\oplus \sigma^{\ast}L$,
    then $V$ is semistable of slope $\frac{1}{2}$ by Lemma \ref{lemma-stability-pullback}.
    Grothendieck's classification of vector bundles on $\mathbb{P}^1$ does not allow for such a bundle,
    see e.g. \cite{vbonp1}.
    
    Consider a point $e\in E$ at which $\pi$ is ramified.
    Let $I$ be the effective Cartier divisor which cuts out $e\in E$.
    Then $I$ is a $G$-invariant subsheaf of $\mathcal{O}_E$ but does not descend to a subsheaf $I'$ of $\mathcal{O}_C$.
    Indeed, by Lemma \ref{lemma-stability-pullback} such a subsheaf $I'$ would be a line bundle of slope $\frac{1}{2}$ which
    is impossible.
\end{example}

\subsection{Functoriality and Small Genus}

\begin{definition}
    A finite group $G$ is called {\it prime to} $p$ if
    $p\nmid \#(G)$.
    A finite separable cover (resp. \'etale cover) $\pi:Y\to X$ of varieties is
    {\it prime to} $p$ if the Galois hull of $\kappa(Y)/\kappa(X)$ (resp. of $Y/X$) has Galois group prime to $p$.
\end{definition}

Observe that prime to $p$ morphisms are well-behaved under composition, i.e., the composition of two such morphisms is again prime to $p$.
We now introduce our functorial notions of stability.
\begin{definition}
\label{definition-fun-stab}
	Let $X$ be a projective variety. 
	A sheaf $V$ on $X$ is called 
	{\it separable-stable}, (resp. {\it \'etale-stable}, resp. \emph{prime to $p$ stable}) 
	if for every finite separable, (resp. finite \'etale, resp. finite \'etale prime to $p$) 
	morphism $\pi:Y\to X$ of varieties the pullback
	$\pi^{\ast}V$ is stable with respect to
	$\pi^{\ast}\mathcal{O}_X(1)$.
\end{definition}

\begin{example}
        Every line bundle is separable-stable. 
        If $p>0$, then a semistable vector bundle
        of rank $r=p^n$, $n\geq 1$, and degree coprime to $p$
        is prime to $p$ stable.
\end{example}
A finite separable morphism has two parts, namely an \'etale part and a genuinely ramified part. We recall the definition:
\begin{definition}
	\label{definition-genuinely-ramified}
	Let $f:Y\to X$ be a cover of varieties.
	We say that $f$ is {\it genuinely ramified} if 
	every factorization $Y\to Y'\to X$ of $f$ such that $Y'\to X$ is an \'etale
	cover satisfies that $Y'\to X$ is an isomorphism.
\end{definition}
Biswas, Das, and Parameswaran show in \cite[Theorem 1.2]{bdp}
that genuinely ramified morphisms of normal projective varieties preserve stability under pullback.
As a direct consequence we obtain:

\begin{cor}
	\label{cor-pro-separable-pro-etale}
	On a normal projective variety the notions of \'etale-stability and
	separable-stability agree for vector bundles.
\end{cor}

\begin{remark}
\label{remark-finite}
Being able to go back and forth between covers and \'etale covers
yields several advantages. 
On the one hand, it is easier to construct Galois covers than \'etale Galois covers.
On the other hand, descent theory is simpler for \'etale Galois covers
and there are - up to isomorphism - only finitely many \'etale covers of fixed degree.
To be precise we have:
\end{remark}

\begin{lemma}
    \label{lemma-finite}
    Let $X$ be a normal projective variety. Then for fixed degree $d$ there
    are only finitely many \'etale covers $Y\to X$ of degree $d$ (up to isomorphism).
\end{lemma}
\begin{proof}
    This is an immediate consequence of the \'etale fundamental group $\pi_{\et}(X)$
    of $X$ being topologically finitely generated.
    To wit, an \'etale cover $Y\to X$ of degree $d$
    corresponds to a finite continuous $\pi_{\'et}(X)$-set of cardinality $d$. 
    Up to isomorphism $S=\{1,\dots,d\}$ and the action of $\pi_{\'et}(X)$
    on $S$ is given by a continuous morphism $\pi_{\'et}(X)\to \mathbf{S}_d$,
    where $\mathbf{S}_d$ denotes the symmetric group of $\{1,\dots,d\}$ equipped with
    the discrete topology.
    As the \'etale fundamental group of a normal projective variety is topologically
    finitely generated, see \cite[Satz 13.1]{popp}, there are only finitely many continuous morphisms to a fixed 
    finite group with the discrete topology.
\end{proof}

The notion of \'etale-stability on a smooth projective curve $C$ is only interesting if $g_C\geq 2$.

\begin{lemma}
	\label{lemma-small-genus}
	Let $C$ be a smooth projective curve of genus $g_C\leq 1$.
	Then the following hold:
	\begin{enumerate}[(i)]
	    \item If $g_C=0$, then the only stable bundles are line bundles.
	    \item If $g_C=1$, then a stable vector bundle of rank $r$ and degree $d$
	            is prime to $p$ stable iff $(r,d)=(1)$ and $r$ is a power of $p$.
	    \item If $g_C=1$ and $C$ is an ordinary elliptic curve, then the only
	        \'etale stable bundles are line bundles.
	    \item If $g_C=1$ and $C$ is supersingular, then the notions of prime to $p$ stable and
	    \'etale stable agree.
	\end{enumerate} 
\end{lemma}
\begin{proof}
	If $g_C=0$, then (i) follows from Grothendieck's classification 
        of vector bundles on $\mathbb{P}^1$, see e.g. \cite{vbonp1}.
	
	In the following we use that semistability is preserved under pullback 
	by a cover and the behaviour of the degree under pullback,
	see Lemma \ref{lemma-stability-pullback}.

	If $g_C=1$, we use \cite[Theorem 5 and Theorem 7]{atiyah}, which are both
	valid in arbitrary characteristic. These theorems
	immediately imply that there are no stable bundles of rank $r>1$ 
        and integral slope over an elliptic curve. 
	In fact more can be said: 
        a semistable vector bundle of rank $r$ and degree $d$
        is stable iff $(r,d)=(1)$, a direct consequence of \cite[Corollary 2.5]{oda}.
    
	Consider a stable bundle $V$ of rank $r > 1$ and degree $d$ 
        such that $(r,d)=(1)$.
	On an \'etale cover of degree non-coprime to $r$ 
	the pullback of $V$ can not be stable by the previous discussion.
	This proves the claim (iii) for ordinary elliptic curves as 
	they have \'etale covers of any square degree. Indeed, for $d$
	not divisible by $p$ multiplication by $d$ is of degree $d^2$. 
	For $d=p$ the dual of the Frobenius $F^{\lor}:E\to E^{(p)}$ is \'etale of
	degree $p$.
	
	If $r$ is a power of $p$ and $(r,d)=(1)$,
	then on all prime to $p$ covers we still have coprime rank and degree.
 	This proves (ii).
 	
	If $C$ is supersingular, then every \'etale cover is prime to $p$ and we obtain (iv). 
\end{proof}

\section{Proof of Theorem 1}
The idea to prove Theorem 1 is simple: There are two types of failure for a stable 
bundle to remain stable after pullback. 
Both of these failures can be detected on single cover. We make this more precise on a smooth projective curve $C$.

The key observation is that a stable bundle $V$ of rank $r$ on $C$
decomposes on an \'etale 
Galois cover $D\to C$ as $V_{\mid D}\cong \bigoplus_{i=1}^n W_i^{\oplus e}$ 
for some pairwise non-isomorphic stable bundles $W_i$ on
$D$ such that the Galois group acts transitively on the isomorphism classes of the $W_i$, see Lemma \ref{lemma-pullback-galois}.
This is somewhat similar to the decomposition of a prime ideal in a Galois extension of number fields; in particular $e$ does not depend on the index $i$.

If $n\geq 2$, this decomposition behaviour can already be detected 
on an \'etale Galois cover $C_{r-large}$, a cover dominating all
\'etale covers of degree dividing $\rk(V)=r$, see Lemma \ref{lemma-large-cover}.

If $V$ remains stable on $C_{r-large}$, then for any \'etale Galois cover 
$D\to C$ the decomposition is $V_{\mid D}\cong W^{\oplus e}$.
Pretending that $W$ descends to a stable bundle $M$ on $C$ 
(this is not clear at all but we provide a technical workaround, see Lemma \ref{lemma-determinant-descend}) 
we can compare the
descent data associated to $M^{\oplus e}$ and $V$ to obtain 
a $\mathrm{Gl}_e$-representation $\rho$ of the Galois group $\mathrm{Gal}(D/C)=G$.
The descent data agree on the kernel of $\rho$ and 
we are reduced to $G$ being a finite subgroup of $\mathrm{Gl}_e$.
If $G$ is prime to $p$, then Jordan's theorem 
- which also has a positive characteristic version due Brauer and Feit -
has a particularly nice form:
\begin{theorem}[\cite{jor} p.114 for characteristic $0$, \cite{jordanp} for positive characteristic]
    \label{theorem-jordan}
    Let $r$ be a natural number, $r \geq 1$. There exists a constant $J(r)$ such that for every finite prime to $p$ subgroup
    $G\subset \mathrm{Gl}_r$ there exists a normal abelian subgroup $N\subseteq G$ of index $\leq J(r)$.
\end{theorem}
Thus, there exists a normal abelian subgroup $N\subseteq G$ of index
$\leq J(e)$, where $J(e)$ denotes the constant from Jordan's theorem. 
As a finite abelian subgroup is simultaneously triagonalizable 
the decomposition $V_{\mid D}\cong W^{\oplus e}$ can already be detected on $D/N$. 
We obtain a prime to $p$ \'etale Galois cover $C_{r-good}$
which detects the stability of $V_{\mid D}$ 
as a cover dominating all prime to $p$ covers of degree $\leq rJ(r)$.

We split the construction of $C_{r-good}$ into two parts. First we show the key lemma and construct $C_{r-large}$.
This construction can also be carried out over any normal projective variety.

Then we continue with the workaround for descending $W$ and finally construct $C_{r-good}$.
The same type of cover works over a normal projective variety $X$. 
However, the workaround for descent only works for curves.
Thus, one has to complete the descent setup on the level of 
$X$ and then restrict the setup to a large curve.

\subsection{A large cover}

The key observation for the (non-)functoriality of stability is the following lemma. 
A stable bundle can only decompose in a very special way after
a Galois pullback.
\begin{lemma}[Key observation]
    \label{lemma-pullback-galois}
    Let $\pi:Y\to X$ be a Galois cover of normal projective varieties with
    Galois group $G$. Let $V$ be a stable vector bundle on $X$ of rank $r$.
    Then $V_{\mid Y}\cong (\bigoplus_{i=1}^n W_i)^{\oplus e}$ for some
    pairwise non-isomorphic stable vector bundles $W_i$ on $Y$ and $n,e\geq 1$.
    Furthermore, $G$ acts transitively on the set of isomorphism classes 
    $\{W_i \mid i=1,\dots, n\}$.
 
    In particular, all the $W_i$ have the same rank $\frac{r}{ne}$. 
\end{lemma}

\begin{proof}
	By Lemma \ref{lemma-stability-pullback} the bundle $V_{\mid Y}$ is
	polystable.
        As $V_{\mid Y}$ is a vector bundle we find that $V_{\mid Y}\cong\bigoplus_{i=1}^n W_i^{\oplus e_i}$ 
	for pairwise non-isomorphic stable vector bundles $W_i$ on $Y$.
	Let $\iota:W\to V_{\mid Y}$ denote the inclusion of one of the $W_i$.
	The image of $\bigoplus_{\sigma\in G}\sigma^{\ast}W
	\xrightarrow{\oplus\sigma^{\ast}\iota}V_{\mid Y}$ is a $G$-invariant subbundle and
	descends to a subbundle $E$ of $V$ by Lemma \ref{lemma-invariant-subsheaf}. 
    As $E$ has the same slope as $V$, the stability of $V$ implies $E=V$. We obtain
    that $\bigoplus_{\sigma\in G}\sigma^{\ast}W\to V_{\mid Y}$ is surjective. 
    Using the stability of the $W_i$ we find that the group $G$ acts
    transitively on the isomorphism classes of the $W_i$. Clearly,
    $\rk(\sigma^{\ast}W)=\rk(W)$ for all $\sigma\in G$.
    
    Let $e=e_{i_0}$ be the smallest index among the $e_i$ and $W=W_{i_0}$.
    For each $W_i$ there is a $\sigma_i\in G$ such that 
    $\sigma_i^{\ast}W\cong W_i$.
    The inclusion $W_i^{\oplus e_i}\to V_{\mid Y}$ induces an inclusion
    $W^{\oplus e_i}\to V_{\mid Y}$ after pullback by $\sigma_i^{-1}$.
    We obtain $e_i\leq e$. By definition of $e$ we have equality.
    The computation of the rank of $W_i$ is now immediate.    
\end{proof}

There are two fundamentally different ways for a stable bundle to decompose on a Galois cover:
$n=1$ or $n\geq 2$ in Lemma \ref{lemma-pullback-galois}. We first find a cover that checks for $n\geq 2$
using that this decomposition can already be seen on a cover of degree $n$.
\begin{lemma}
	\label{lemma-decomposition-small-degree} 
	Let $\pi:Y\to X$ be a Galois cover of normal projective
	varieties with Galois group $G$.
	Further, let $V$ be a stable vector bundle of rank $r$ on $X$ such that 
	the decomposition $V_{\mid Y}\cong \bigoplus_{i=1}^n W_i^{\oplus e}$ 
	of Lemma \ref{lemma-pullback-galois} satisfies $n\geq 2$.  
	Then there is a factorization of $Y\to X$ into $Y\to
	Y'\xrightarrow{\pi'}X$ 
	such that $\text{deg}(\pi')=n$ and $V_{\mid Y'}$
	is not stable.
	
	More precisely, $V_{\mid Y'}\cong V'\oplus W'$, where $W'$ is of rank $r/n$
	and $V'_{\mid Y}$ is isomorphic to a direct sum of
	conjugates of $W'_{\mid Y}$ under $G$.  
\end{lemma}

 \begin{proof} 
	By assumption there are at least
	two different $W_i$.
	Consider the stabilizer $H$ of $W:=W_i^{\oplus e}$ for some $i$
	and fix an inclusion $\iota:W\to V_{\mid Y}$.
	The image $E$ of $\bigoplus_{\sigma\in H}\sigma^{\ast}W
	\xrightarrow{\oplus \sigma^{\ast}\iota} V_{\mid Y}$ is
	an $H$-invariant subsheaf. Using the stability of the $W_j$ we find that $E$ 
	is isomorphic to $W$. Therefore, the direct summand 
	$W$ of $V_{\mid Y}$
	descends to a direct summand $W'$ of $V_{\mid Y'}$, where $Y'=Y/H$ and
	$Y\to Y'\xrightarrow{\pi'} X$ are the induced
	morphisms.  Note that $\pi'$ has degree $\#(G/H)=n$.

	Let $V_{\mid Y'}\cong W'\oplus V'$.
	As $G$ acts transitively on the isomorphism classes of the $W_i$ 
	we have that $V'_{\mid Y}$ is a direct sum of
	$\sigma^{\ast}W'_{\mid Y}$ for some $\sigma\in G$.  
\end{proof}

As a direct consequence we obtain the large cover checking for 
decomposition of a stable bundle into at least two non-isomorphic stable bundles on some cover:
\begin{lemma}
	\label{lemma-large-cover}
	Let $X$ be a normal projective variety and $r\geq 2$.  
	Then we have the following:
	\begin{enumerate}[(i)]
		\item There exists an \'etale Galois cover
			$X_{r-large}\to X$ satisfying the following: 
   
			If $V$ is a vector bundle of rank $r$ on $X$ such that
			$V_{\mid X_{r-large}}$ is stable, then 
			for all \'etale Galois covers $Y\to X$
			we have $V_{\mid Y}\cong W^{\oplus e}$ 
                for some stable vector bundle $W$ on $Y$
			and $e\geq 1$.  
		\item There is an \'etale prime to $p$ Galois cover
			$X'_{r-large}\to X$
			such that:
   
			If $V$ is a vector bundle of rank $r$ on $X$ such that
			$V_{\mid X'_{r-large}}$ is stable,
			then for all \'etale prime to $p$ Galois covers $Y\to X$
			we have $V_{\mid Y}\cong W^{\oplus e}$ for some 
                stable vector bundle $W$ on $Y$
			and $e\geq 1$.
	\end{enumerate}
\end{lemma}

\begin{proof}
	(i): Decomposing into different stable vector bundles 
        descends to some \'etale cover of
	degree $n$ such that $n\mid r$, 
        see Lemma \ref{lemma-decomposition-small-degree}.
	There are only finitely many such
	\'etale covers up to isomorphism, see Lemma \ref{lemma-finite}.
        In particular, there is an \'etale Galois cover
	$X_{r-large}$ dominating all
	\'etale covers of degree dividing $r$. 
	This is the desired cover.

	(ii): Define $X'_{r-large}$ as an \'etale prime to $p$ Galois cover
	dominating all \'etale prime to $p$ covers
	of degree dividing $r$. This is the desired cover.
\end{proof}

\subsection{A good cover}
To construct the cover $X_{r-good}$ detecting prime to $p$ stability it remains to deal with decomposition behaviour of the
form $V_{\mid Y}=W^{\oplus e}$, where $Y\to X$ is a Galois cover of normal projective varieties and $V$ a stable bundle.
We start with the workaround for descent of $G$-invariant stable bundles. This requires working on curves and the mild assumption that 
$\det(W)$ already descends. The determinant-descent can be set up on arbitrary varieties and we are then able to derive the main theorem
by reducing to the case of curves via a restriction theorem for stability.

We start with the workaround for descent.
If one is only interested in the case of curves, 
then there is an honest descent lemma one could use instead, 
see Lemma \ref{lemma-jochen}. The workaround roughly says that a $G$-linearization of the determinant of an $G$-invariant simple bundle can lifted to a linearization
for a slightly bigger Galois cover.

\begin{lemma}[Workaround for descent]
	\label{lemma-determinant-descend}
	Let $D\to C$ be a Galois cover of smooth projective curves with
	Galois group $G$.
	Let $V$ be a simple $G$-invariant vector bundle of rank $r$ on $D$.
	Further, assume that $\det(V)$ admits a $G$-linearization.
	
	Then there exists a lift of the $G$-linearization of $\det(V)$ 
	to a system of isomorphisms $\psi_{\sigma}:V\xrightarrow{\sim} \sigma^{\ast}V$.
	Furthermore, there exists a cyclic Galois cover $\varphi:D'\to D$
	such that
	\begin{enumerate}[(i)]
	    \item $\varphi$ is prime to $p$ of degree $\deg(\varphi)\mid r$,
	    \item $D'\to D\to C$ is a Galois cover,
	    \item $\mathrm{Gal}(D'/D)\subseteq \mathrm{Gal}(D'/C)$ is central, and
	    \item there exists a $1$-cocycle $\alpha: \mathrm{Gal}(D'/C)\to \mu_{r}$ such that 
	         \[
	            \varphi^{\ast}\big(\psi_{\sigma}\big)\cdot \alpha(\sigma')^{-1}:
	            V_{\mid D'} \xrightarrow{\sim} \sigma'^{\ast}V_{\mid D'}
	        \]
	          defines a $\mathrm{Gal}(D'/C)$-linearization of $V_{\mid D'}$, 
	          where $\sigma$ denotes the image of $\sigma'$ under the natural morphism $\mathrm{Gal}(D'/C)\to G$.
	\end{enumerate}
\end{lemma}

\begin{proof}
	For two simple isomorphic bundles $V$ and $W$ we have a surjective morphism
	$\Hom(V,W)\xrightarrow{\det}\Hom(\det(V),\det(W))$.
	Thus, the $G$-linearization
	of $\det(V)$ lifts to isomorphisms $\psi_{\sigma}:V\xrightarrow{\sim}\sigma^{\ast}V$
	such that $\psi_{\sigma\tau}^{-1}\circ\tau^{\ast}\psi_{\sigma}\circ \psi_{\tau}=\lambda_{\sigma,\tau}\in \mu_r$.
	Indeed,	after identifying $\Hom(V,V)$
	with $k$ the determinant corresponds to the $r$-th power map.
	
	A computation \cite[Proposition 2.8]{dk}
	shows that the family $\lambda_{\sigma,\tau}$ defines a $2$-cocycle.
	Let $p^n r'=r$ with $r'$ coprime to $p$ and $\lambda'_{\sigma,\tau}=\lambda_{\sigma,\tau}^{p^n}$.
	The $2$-cocycle condition for $\lambda_{\sigma,\tau}$ implies the
	$2$-cocycle condition for $\lambda'_{\sigma,\tau}$.
	We obtain
	$\lambda':=(\lambda'_{\sigma,\tau})\in H^2(G,\mu_{r'}).$

	Let $\mathrm{Gal}$ be the absolute Galois group of $\kappa(C)$.
	As $C$ is a curve over an algebraically closed field, 
        $\kappa(C)$ is a $C_1$ field
	by Tsen's Theorem, see \cite[Corollary 6.5.5]{nsw}.
	In particular, $H^2(\mathrm{Gal},(\kappa(C)^{sep})^*)$ vanishes, 
        see \cite[Proposition 6.5.8]{nsw}.
	By Hilbert 90 we also have 
        vanishing of $H^1(\mathrm{Gal},(\kappa(C)^{sep})^*)$, 
        see \cite[Theorem 6.2.1]{nsw}.
	Applying these two vanishing results to the long exact
	cohomology sequence
	of the short exact sequence
	\[
	    0 \to \mu_{r'}\to (\kappa(C)^{sep})^*
	    \xrightarrow{x\mapsto x^{r'}}(\kappa(C)^{sep})^* \to 0
	\]
	we obtain $H^2(\mathrm{Gal}, \mu_{r'})=0$.

	By \cite[Theorem 1.2.4]{nsw} the element 
        $\lambda'\in H^2(G,\mu_{r'})$ corresponds to an	extension 
	\[
	    0\to \mu_{r'}\to G'\to G \to 0
	\]
	inducing the action of $G$ on $\mu_{r'}$. As the action of $G$ on
	$\mu_{r'}$ is trivial, we find that $\mu_{r'}$ is central in $G'$.
	Write $G$ as a quotient of $\mathrm{Gal}$. 
	Since $H^2(\mathrm{Gal},\mu_{r'})=0$, we obtain that the central extension 
	\[
	    0\to \mu_{r'}\to \mathrm{Gal}\times_G G'\to \mathrm{Gal} \to 0,
	\]
	is trivial, i.e., $\mathrm{Gal}\times_G G'\cong\mathrm{Gal}\times \mu_{r'}$.
	In particular, there exists a surjection $\mathrm{Gal}\times  \mu_{r'}\to G'$.
	Let $H$ denote the image of $\mathrm{Gal}\times 0$ under this morphism.
	By construction $H\to G'\to G$ is surjective. As $H\subseteq G'$
        we find that
        \[
        0\to \mu_{r'}\to H\times_G G'\to H \to 0
        \]
        is a central split extension
        and thus trivial.
        
	The kernel $K$ of $H\twoheadrightarrow G$ is a subgroup of $\mu_{r'}$.
	In particular, $K\subseteq H$ is central and cyclic.
	Denote by $\kappa(D')$ the field extension of $\kappa(C)$ corresponding
	to $\mathrm{Gal}\twoheadrightarrow H$ and by $D'$ the associated curve.
	We obtain Galois covers
	$D'\xrightarrow{\varphi} D\to C$ such that 
	$\mathrm{Gal}(D'/D)\subseteq \mathrm{Gal}(D'/C)$ is central
	and cyclic. Furthermore, the obstruction $\lambda'\in H^2(G,\mu_{r'})$
	vanishes in $H^2(H,\mu_{r'})$. 
	
	The triviality of the $2$-cocycle $\varphi^{\ast}\lambda'\in
	H^2(H,\mu_{r'})$
	means that there is a 1-cocycle $\alpha':H\to \mu_{r'}$ such that 
	$\partial(\alpha')(\sigma,\tau)=\lambda'_{f(\sigma),f(\tau)}$,
	where $f:H\twoheadrightarrow G$ denotes the surjection constructed above.  
	
	Recall that in positive characteristic $p$-th roots are unique.
	Thus, there is a 1-cocycle $\alpha:H\to \mu_{r},\sigma\mapsto
	\alpha'(\sigma)^{1/p^n}$ such that $\partial(\alpha)(\sigma,\tau)=\lambda_{f(\sigma),f(\tau)}$.
	By construction the isomorphisms $\varphi^{\ast}\psi_{f(\sigma)}\cdot
	\alpha(\sigma)^{-1}, \sigma\in H,$ define a linearization.
	Indeed, we have
	\begin{align*}
	\big(\varphi^{\ast}\psi_{f(\sigma\tau)}\cdot \alpha(\sigma\tau)^{-1}\big)^{-1}\circ	
        \tau^{\ast}\varphi^{\ast}\psi_{f(\sigma)}\cdot \alpha(\sigma)^{-1}\circ
	\varphi^{\ast}\psi_{f(\tau)}\cdot \alpha(\tau)^{-1}
		 & =\\
		\lambda_{f(\sigma),f(\tau)}\cdot(\partial(\alpha)(\sigma,\tau))^{-1} &  =1.
	\end{align*} 
\end{proof}

\begin{remark}
	A shorter (but less precise) argument is the following:
	Recall that $H^2(\mathrm{Gal},\mu_{r'})=\colim H^2(G',\mu_{r'})$, 
        see \cite[Proposition 1.2.5]{nsw}, where
	the colimit is taken over all finite Galois extensions of $\kappa(C)$ and
	$G'$ denotes the Galois group. We obtain $\mathrm{Gal}\twoheadrightarrow
	G'\twoheadrightarrow G$ such that
	the obstruction $\lambda$ vanishes on the associated curve. However,
	this does not give us a way to control the kernel which is crucial.
\end{remark}

Note that Lemma \ref{lemma-determinant-descend} only works for curves 
and requires the mild assumption that the determinant descends.
Given that the determinant descends we can detect decomposition on a cover of degree bounded by the constant of
Jordan's theorem. We do this in the following lemma. This would already allow us to deduce Theorem 1 for curves
but we only give the general proof later.
\begin{lemma}
    \label{lemma-invariant-inclusion}
    Let $D\to C$ be an \'etale prime to $p$ Galois cover with Galois group $G$.
    Let $V$ be a vector bundle on $C$ such that $V_{\mid D}\cong W^{\oplus e}$ for some
    simple $G$-invariant vector bundle $W$ satisfying that $\det(W)$ descends to $C$.
    Denote the constant from Jordan's theorem, 
    see Theorem \ref{theorem-jordan}, by $J(e)$.
    
    Then there exists a normal subgroup $N\subseteq G$ of index $\leq J(e)$ 
    and $W'\subseteq V_{\mid C'}$ such that $W'_{\mid D}\cong W$, where
    $D\rightarrow C':=D/N \rightarrow C$ are the natural morphisms.
\end{lemma}
\begin{proof}
    Denote the rank of $W$ by $r$.
    Let $\psi^W_{\sigma}:W\xrightarrow{\sim} \sigma^{\ast}W, \sigma\in G$, be a system of isomorphisms lifting the
    descent datum of $\det(W)$, see Lemma \ref{lemma-determinant-descend}.
    By the same lemma there is a Galois cover $D'\xrightarrow{\varphi} D$ with prime to $p$ 
    cyclic Galois group $H$ such that $D'\to D\to C$ is a Galois cover with Galois group $G'$.
    Further, there exists a $1$-cocycle $\alpha:G'\to \mu_r$ 
    such that $\varphi^{\ast}(\psi^W_{\sigma})\cdot\alpha(\sigma')^{-1}$
    is a $G'
    $-linearization, where $\sigma$ denotes the image of $\sigma'$ in $G$.
    Furthermore, $H\subseteq G'$ is central.
    
    Our goal is to find a normal subgroup $N'\subseteq G'$ of index $\leq J(e)$ containing $H$ 
    and an $N'$-invariant subbundle $W_{\mid D'}\subseteq V_{\mid D'}$.
    By Lemma \ref{lemma-invariant-subsheaf} the inclusion $W_{\mid D'}\subseteq V_{\mid D'}$ 
    descends to $C'$, where $C'$ denotes the normal closure of $C$ in 
    the fixed field
    $\kappa(D')^{N'}$. Then the lemma follows as $C'=C/N$, where $N$ is the image of $N'$ in $G$.
    
    Let $\psi^V_{\sigma}:V_{\mid D}\xrightarrow{\sim} \sigma^{\ast}V_{\mid D}$ be the descent datum associated to $V$.
    Choose an isomorphism $\psi:V_{\mid D}\xrightarrow{\sim} W^{\oplus e}$ which exists by assumption.
    Define a map 
	\[
	    \rho:G'\to \mathrm{Gl}_{e}, \sigma'\mapsto
	\text{diag}(\alpha(\sigma'))
	((\psi_{\sigma}^{W})^{-1})^{\oplus e}\circ
	\sigma^{\ast}(\psi)\circ\psi^{V}_{\sigma}\circ\psi^{-1},
	\]
	where $\sigma$ denotes the image of $\sigma'$ in $G$, i.e., $\rho$ measures the failure
	of the following diagram
	\[
	    \begin{tikzcd}
	        W^{\oplus e}  & &
	        V_{\mid D} \ar[d, "\psi^{V}_{\sigma}"]\ar[ll, "\psi"']\\
	        \sigma^{\ast}W^{\oplus e} \ar[u, "((\psi^{W}_{\sigma})^{-1})^{\oplus e}"]
	        &\ & \sigma^{\ast}V_{\mid D} \ar[ll, " \sigma^{\ast}(\psi)"]
	    \end{tikzcd}
	\]
	to commute twisted by $\mathrm{diag}(\alpha(\sigma'))$.
	Another way to put this is that $\rho$ compares the $G'$-linearizations
	$(\varphi^{\ast}(\psi^{W}_{\sigma})^{-1})^{\oplus e}\mathrm{diag}(\alpha(\sigma'))$ and $\varphi^{\ast}(\psi^{V}_{\sigma})$
	on $D'$.
	
	We claim that $\rho$ defines a group morphism.
	Indeed, for $\sigma',\tau'\in G'$ mapping to $\sigma$ (resp. $\tau$)
	in $G$ we have
	\begin{align*}
		\rho(\tau')\rho(\sigma') & = \ \\
	\text{diag}(\alpha(\tau'))
		((\psi_{\tau}^{W})^{-1})^{\oplus e}
	\tau^{\ast}(\psi)\psi^{V}_{\tau}\psi^{-1}
	\text{diag}(\alpha(\sigma'))
		((\psi_{\sigma}^{W})^{-1})^{\oplus e}
		\sigma^{\ast}(\psi)\psi^{V}_{\sigma}\psi^{-1} & = \ \\
		\text{diag}(\alpha(\tau')\alpha(\sigma'))
		((\psi_{\tau}^{W})^{-1})^{\oplus e}
	\tau^{\ast}(\psi)\psi^{V}_{\tau}\psi^{-1} 
		((\psi_{\sigma}^{W})^{-1})^{\oplus e}
		\sigma^{\ast}(\psi)\psi^{V}_{\sigma}\psi^{-1} & =\ \\
		\text{diag}(\alpha(\tau')\alpha(\sigma'))
		((\psi_{\sigma}^{W})^{-1})^{\oplus e}
		\sigma^{\ast}{\bigg (}((\psi_{\tau}^{W})^{-1})^{\oplus e}
		\tau^{\ast}(\psi)\psi^{V}_{\tau}\psi^{-1}{\bigg )} 
		\sigma^{\ast}(\psi)\psi^{V}_{\sigma}\psi^{-1} & = \ \\
		\text{diag}(\alpha(\tau')\alpha(\sigma'))
		((\psi_{\sigma}^{W})^{-1})^{\oplus e}
		\sigma^{\ast}((\psi_{\tau}^{W})^{-1})^{\oplus e}
		\sigma^{\ast}\tau^{\ast}(\psi)\sigma^{\ast}(\psi^{V}_{\tau})
		\psi^{V}_{\sigma}\psi^{-1} & = \ \\
			\text{diag}(\alpha(\tau'\sigma'))
		((\psi_{\tau\sigma}^{W})^{-1})^{\oplus e}
		(\tau\sigma)^{\ast}(\psi)\psi^{V}_{\tau\sigma}
		\psi^{-1}  & = \ \\
		\rho(\tau'\sigma'), & \ 
	\end{align*}
	where only the third and fifth equality require an explanation. 
	To obtain the third equality we use that $((\psi_{\sigma}^{W})^{-1})^{\oplus e}$ commutes with matrices 
	and that matrices with entries in $k$ do not change under pullback.
	To obtain the fifth equality we note that by construction of $\alpha$
        the isomorphisms
        $\varphi^{\ast}(\psi^{W}_{\sigma})\alpha(\sigma')^{-1}$ define a $G'$-linearization, 
        see Lemma \ref{lemma-determinant-descend}.
	
	Replacing $D'$ by $D'/\ker(\rho)$ we can assume that $G'$ is a subgroup of $\mathrm{Gl}_{e}$.
	By Jordan's theorem, see Theorem \ref{theorem-jordan}, there is a normal abelian subgroup $N'\subseteq G'$
	such that $G'/N'$ has cardinality at most $J(e)$. As $H$ is central in $G'$ the subgroup $N'+H$ is normal, abelian, and 
	contains $H$. As a finite
	abelian subgroup of $\mathrm{Gl}_{e}$ is simultaneously triagonalizable,
	we find the desired	
	$(N'+H)$-invariant inclusion $W_{\mid D'}\subseteq V_{\mid D'}$. 
\end{proof}

    To be able to apply the previous lemma
    we need to find a way to descend the determinant bundle. 
    For such a construction we need to take roots of line bundles. 
    If we avoid the characteristic, 
    then this is always possible up to a cyclic cover.
\begin{lemma}
\label{lemma-root-preparation}
	Let $X$ be a normal projective variety. Let $d$ be an integer prime to $p$.
	Further, let $L$ be a line bundle on $X$.
	Then there exists a cyclic Galois cover $\varphi:X'\to X$ such that
	$\deg(\varphi)\mid d$ and
	$L_{\mid X'}$ admits a $d$-th root on a big open subscheme.
\end{lemma}

\begin{proof}
	Let $\mathcal{O}_X(1)$ be an ample line bundle. Clearly, it suffices to find a morphism $X'\to X$ 
	as in the statement such that $L_{\mid X'}\otimes \mathcal{O}_{X'}(1)^{\otimes Nd}$ has a $d$-th root for some $N$.
	Thus, we can assume that $L$ admits a non-zero global section, i.e.,
	$L=\mathcal{O}_X(D)$ for some effective Cartier divisor $D$. Observe
	that it suffices to prove the Lemma for $\mathcal{O}_X(-D)$ instead of $L$.
	
	Choose an affine open $U$ containing the generic point of $D$ in $X$ such
	that $D_{\mid U}=V(f)$ for some non-zero divisor $f\in \mathcal{O}_U$.
	Consider the field extension $K/\kappa(X)$ generated by a $d$-th root of $f$.
	As $p\nmid d$ the extension $K/\kappa(X)$ is cyclic of order $d'\mid d$.
	Let $X'$ denote the normalization of $X$ in $K$. Note that
	there is a canonical finite morphism $\varphi:X'\to X$ of normal projective varieties.
	It is also separable by construction. As we only want to find an a $d$-th root on a big open
        and $X'\to X$ is flat at all codimension $1$ points, we can assume that $X'\to X$ is flat.
	
	Consider $U':=\varphi^{-1}(U)\cup \varphi^{-1}(X\setminus D)$. 
        By construction $U'$ is big.
	We show that 
	$\mathcal{O}_{X}(-D)_{\mid X'}$ admits a $d$-th root on $U'$.
	Let $t$ be a $d$-th root of $f$ on $\varphi^{-1}(U)$.
	Then $t$ defines an effective Cartier divisor $D'$ on $U'$. 
	We have $\mathcal{O}_{U'}(-D')^{\otimes d}=\mathcal{O}_{X}(-D)_{\mid U'}$ as $t^d=f$ on $\varphi^{-1}(U)$ by construction
	and both are trivial on $\varphi^{-1}(X\setminus D)$.
\end{proof}

\begin{definition}
    A morphism $\pi:Y\to X$ of varieties is called {\it quasi-\'etale}
    if there is some big open subset $U\subseteq X$ such that $\pi^{-1}(U)\to U$ is \'etale.
    If $\pi^{-1}(U)\to U$ is an \'etale Galois cover with Galois group $G$, 
    then we also say that $\pi:Y\to X$
    is a \emph{quasi-\'etale Galois cover} with Galois group $G$.
\end{definition}

We can now set up the determinant descent needed to apply
Lemma \ref{lemma-invariant-inclusion} on a normal projective variety.

\begin{lemma}
	\label{lemma-descend-preparation-higher-dim}
	Let $X$ be a normal projective variety.
	Let $Y\to X$ be an \'etale prime to $p$ Galois cover with Galois group $G$.
	Further, let $V$ be a stable vector bundle of rank $r$ on $X$
	such that $V$ is stable on $X'_{r-large}$.
 	Then there exists a commutative diagram
	of normal projective varieties
	\[
	\begin{tikzcd}
		Y' \ar[r] \ar[d] & X' \ar[d]\\
		Y \ar[r] & X
	\end{tikzcd} 
	\ \ \ \text{such that}
	\]
	\begin{enumerate}[(i)]
		\item we have $V_{\mid Y'}\cong W'^{\oplus e'}$ such that $W'$ is stable and
	          $\det(W')$ descends along $Y'\to X'$ on some big open subscheme of $Y'$,
		\item $Y'\to X$ is a prime to $p$ Galois cover,
		\item $X'\to X$ is cyclic of degree dividing $r$, and
	    \item[(iv)] $Y'\to X'$ is a quasi-\'etale Galois cover.
	\end{enumerate}
\end{lemma} 

\begin{proof}
	Consider the decomposition $V_{\mid Y}\cong W^{\oplus e}$ of Lemma
	\ref{lemma-pullback-galois}. Clearly, $\det(W)^{\otimes e}$ and
	$\bigotimes_{\sigma \in G}\sigma^{\ast}\det(W)\cong\det(W)^{\otimes
	\#(G)}$ descend to $X$. Therefore, $\text{det}(W)^{\otimes d}$ descends to $X$ as well,
	where $d=\mathrm{gcd}(e,\#(G))$. Thus, there exists a line bundle $L$
	on $X$ such that $L_{\mid Y}\cong\text{det}(W)^{\otimes d}$.  
	Note that $p\nmid d$ since $G$ is prime to $p$. 
	
	We can apply Lemma \ref{lemma-root-preparation} to find $X'\to X$
    such that $L_{\mid X'}$ has a $d$-th root $L'$ on a big open $U'$ of $X'$.
    Consider a connected component $Y''$ of the normalization of the reduced fibre product $(Y\times_X X')_{red}$.
    Note that the natural morphism $\psi:Y''\to X'$ is prime to $p$ and Galois.
    Then 
    \[
        W'':=\det(W)_{\mid \psi^{-1}(U')}
        \otimes L'^{-1}_{\mid \psi^{-1}(U')}
    \]
    is a line bundle of order dividing $d$. The spectral cover 
    $U'''\to \psi^{-1}(U')$ associated to $W''$ trivializes $W''$.
    
    Let $Y'$ denote the normalization of $Y''$ in $K$, where $K$ is the Galois hull of $\kappa(U''')/\kappa(X)$.
    As $\kappa(U''')/\kappa(Y''), \kappa(Y'')/\kappa(X'),$ and $\kappa(X')/\kappa(X)$ are prime to $p$
    the same holds for $\kappa(Y')/\kappa(X)$.
    Then the commutative diagram
    \[
	\begin{tikzcd}
		Y' \ar[r] \ar[d] & X' \ar[d]\\
		Y \ar[r] & X
	\end{tikzcd}
	\]
    satisfies the conditions (ii), (iii), and (iv) of the Lemma.
    
    If $W_{\mid Y'}$ is stable, then $V_{\mid Y'}\cong W_{\mid Y'}^{\oplus e}$ and we obtain (i) by construction.
    If $W_{\mid Y'}$ is not stable, then
    we repeat the above construction
    replacing $Y$ by the \'etale part of $Y'/X$.
    Then we have $V_{\mid Y}\cong W'^{\oplus e'}$ for $e'> e$ and $W'$ stable.
    As the integer $e'$ is at most $r$,
    this process stops after finitely many iterations.
\end{proof}

We can now prove the main theorem. 
\begin{theorem}
	\label{theorem-very-large-cover}
	Let $X$ be a normal projective variety of dimension at least $1$.
	Let $r\geq 2$.
	Then there exists
	an \'etale prime to $p$ Galois cover $X_{r-good}\to X$
	such that a vector bundle $V$ of rank $r$ on $X$ is
    prime to $p$ stable iff $V_{\mid X_{r-good}}$ is stable.
	
	In particular,
	prime to $p$ stability is an open property
	in the moduli space of Gieseker semistable sheaves on $X$.
\end{theorem}

\begin{proof}
    Let $X_{r-good}$ be an \'etale prime to $p$ Galois cover dominating $X'_{r-large}$
    from Lemma \ref{lemma-large-cover} and all prime to $p$ covers of degree $\leq J(r)r$, where $J(r)$
	is the bound from Jordan's theorem, see Theorem \ref{theorem-jordan}.

    The "only if" part is trivial.
    For the "if" part let $V$ be a vector bundle of rank $r$ on $X$ such that $V_{\mid X_{r-good}}$ is stable. 
    Consider an \'etale prime to $p$ Galois cover
	$Y\to X$ and let $V_{\mid Y}\cong W^{\oplus e}$ be the decomposition
	of Lemma \ref{lemma-pullback-galois}.
	Applying Lemma \ref{lemma-descend-preparation-higher-dim}
	we obtain a commutative diagram
	\[
	\begin{tikzcd}
		Y'\ar[r] \ar[d] & X'\ar[d]\\
		Y \ar[r] & X
	\end{tikzcd}
	\]
	satisfying the properties (i) - (iv) of Lemma
	\ref{lemma-descend-preparation-higher-dim}. 
	In particular, we have an isomorphism $V_{\mid Y'}\cong W'^{\oplus e'}$ for some stable bundle $W'$
	such that $\det(W')$ descends on some big open along $Y'\to X'$.

    Observe that $V':=V_{\mid X'}$ is stable
    as the degree of $X'\to X$ is at most $r$.
    
    By Bertini's theorem the general complete intersection curve $C'$ in $X'$ is irreducible
    and irreducible after pullback to $Y'$, see \cite[Corollaire 6.11 (3)]{jou}.
    Furthermore, the general such $C'$ is also normal by \cite[Theorem 7]{sei}.
    The general hyperplane section intersects the locus where $Y'\to X'$ is not \'etale transversally.
    As $Y'\to X'$ is quasi-\'etale we obtain that the pullback $D'$ of the general such $C'$ is an \'etale cover of $C'$.
    We also note that $D'\to C'$ is an \'etale Galois cover with the same Galois group as $Y'\to X'$.
    
    Observe that there are only finitely many intermediate 
    quasi-\'etale Galois covers $Y'\to Y''\to X'$, where $Y''$ is a normal projective variety.
    On $Y''$ the bundle $V'$ decomposes as $V'_{\mid Y''}\cong W''^{\oplus e''}$ for some stable bundle $W''$ on $Y''$.
    Iterating the restriction theorem in 
    arbitrary characteristic for normal projective varieties, 
    see \cite[Theorem 0.1]{la24} for positive characteristic and \cite[Theorem 7.17]{yeh} for arbitrary characteristic,
    we find that restricting $W''$ to $D'':=Y''\times_{X'}C'$ is stable, where $C'$ is a 
    general complete intersection curve
    in $c_1(\mathcal{O}_{X'}(-N_1))\dots c_1(\mathcal{O}_{X'}(-N_{n-1}))$ for $N_i\gg 0$.
	
    Restricting the decomposition $V'_{\mid Y'}\cong W'^{\oplus e'}$ of $V'$ on $Y'$ 
    to such a $D':=Y'\times_{X'}C'$ we obtain an isomorphism
    $(V'_{\mid C'})_{\mid D'}\cong (W'_{\mid D'})^{\oplus e'}$.
    Note that $W'_{\mid D'}$ is stable and for general $C'$ its determinant 
    $\det(W'_{\mid D'})$ descends to $C'$ by property (i) of Lemma \ref{lemma-descend-preparation-higher-dim}. 
    Hence, we are in a position to apply Lemma \ref{lemma-invariant-inclusion}.
    Thus, there is an intermediate cover $D' \to D''\to C'$ of degree $\leq J(e')$
    such that there is a stable subbundle $M''\subseteq V'_{\mid D''}$ pulling back to $W'_{\mid D'}$ on $D'$.
	
    The intermediate cover $D' \to D''\to C'$ can be lifted to a quasi-\'etale factorization of $Y'\to Y''\to X'$.
    Indeed, let $K$ be the kernel of the natural morphism $\mathrm{Gal}(D'/D'')\to \mathrm{Gal}(D'/C')$.
    As $\mathrm{Gal}(Y'/X')=\mathrm{Gal}(D'/C')$
    we can define $Y''$ to be the normalization of $X'$ in the field extension $\kappa(Y')^{K}/\kappa(X')$.
    
    Note that $Y''\to X$ is prime to $p$ of degree at most $rJ(e')\leq rJ(r)$.
    Consider the factorization $Y'' \to Y'''\to X$ into its \'etale and
    genuinely ramified part.
    We find that $V_{\mid Y'''}$ is stable by assumption. By 
    \cite[Theorem 2.5]{bdp} genuinely ramified covers preserve stability and the bundle $V_{\mid Y''}=V'_{\mid Y''}$ is stable as well.
    Thus, $V'_{\mid Y''}\cong W''$ and we obtain the stability of $V'_{\mid D''}$. Therefore, $V'_{\mid D''}\cong M''$ and
    pulling back to $D'$ we find $V'_{\mid D'}\cong W'_{\mid D'}$, i.e., $e'=1$.
    Clearly, $e\leq e'$ and we conclude that $V_{\mid Y}$ is stable.
\end{proof}

\begin{remark}
    We can interpret Theorem \ref{theorem-very-large-cover} in terms of prime to p \'etale trivializable bundles of rank $r\geq 2$, i.e., bundles that become trivial after pullback to some \'etale prime to $p$ Galois cover. 
    Such bundles correspond to a $\gl_r$-representation of the prime to $p$ completion $\pi_{\et}'(X)$. Moreover, stable prime to $p$ \'etale trivializable bundles correspond to irreducible representations of $\pi_{\et}'(X)$. Then Theorem \ref{theorem-very-large-cover}
    says that such an irreducible representation of rank $r$ becomes reducible after restricting along $\pi_{\et}'(X_{r,good})\to \pi_{\et}'(X)$.

    For such a statement an \'etale prime to $p$ Galois cover dominating all
    \'etale covers of degree bounded by the constant of Jordan's theorem would
    suffice. Indeed, any representation of $\rho:\pi'_{\et}(X)\to \gl_r$ factors
    via a finite prime to $p$ subgroup $G$ of $\gl_r$ such that $G$ is a
    quotient of $\pi'_{\et}(X)$. By (the analogue of) Jordan's theorem, Theorem
    \ref{theorem-jordan}, there exists a finite abelian subgroup $N$ of $G$ of
    index at most $J(r)$. Then $G$ corresponds to an \'etale Galois cover $Y\to
    X$ with Galois group $G$ and the subgroup $N$ corresponds to an intermediate
    \'etale Galois cover $Y\to Y/N\to X$.
    By construction the restriction of the representation $\rho$ to
    $\pi_{\et}'(Y/N)$ becomes reducible as the abelian group $N$ does not admit
    irreducible representation of degree $r\geq 2$.

    We summarize the argument in a commutative diagram of \'etale Galois covers
    \[
    \begin{tikzcd}
        \ & \ & X_{r,good} \ar[d] \ar[dl,dotted] \\
        Y\ar[r, "N"] & Y/N \ar[r,"G/N"] & X,
    \end{tikzcd}
    \]
    where we obtained the horizontal factorization via Jordan's theorem and
    the dotted arrow using the finiteness of \'etale covers of bounded degree.
\end{remark}

\section{Proof of Theorem 2}
Consider a smooth projective curve $C$ of genus $g_C\geq 2$.
To obtain the non-emptiness of the locus of prime to $p$ stable bundles $M^{p'-s,r,d}_C$
we find estimates for the dimension of the complement 
\[
    Z:=M^{s,r,d}_C \setminus M^{p'-s,r,d}_C.
\]
This complement decomposes into two strata $Z=Z_1 \sqcup Z_2$, where
\[
    Z_1:=\{ V\in M^{s,r,d}_C \mid V_{\mid C_{r-good}}\cong W^{\oplus e}, W \text{ stable on } C_{r-good}, e\geq 2\} \text{ and}
\]
\[
    Z_2:=\{ V\in M^{s,r,d}_C \mid V_{\mid C_{r-good}}\cong \bigoplus_{i=1}^n W_i^{\oplus e}, W_i \text{ stable on } C_{r-good}, n\geq 2\}
\]
are obtained via applying Lemma \ref{lemma-pullback-galois} to $C_{r-good}\to C$.

To this end we first reprove a theorem due to Faltings asserting that
pullback by a cover induces a finite morphism on the level of moduli spaces of semistable vector bundles. This gives us the flexibility 
to compute the dimension after such a pullback.

Finding an estimate for $\dim(Z_2)$ is fairly simple: the transitive action of the Galois group
allows us to essentially recover the decomposition $V_{| C_{r-good}}\cong\bigoplus_{i=1}^n W_i^{\oplus e}$
from a semistable vector bundle $W'$ on an intermediate cover $D'\to C$ of degree $n$.

To find an estimate for $\dim(Z_1)$ one has to compare the notions of $G$-linearization and $G$-invariance.
While a $G$-invariant vector bundles might not descend,
a simple $G$-invariant bundle does so up to twist by a line bundle.

\subsection{Pullback is finite}

\begin{lemma}
	\label{lemma-etale-pushforward}
	Let $\pi:D\to C$ be an \'etale cover of smooth projective curves.
	Then we have the following:
	\begin{enumerate}[(i)]
		\item The pushforward of a semistable bundle on $D$ to $C$ is
			semistable.
		\item Let $V$ be a semistable vector bundle on $C$.
			Then $\pi_{\ast}\mathcal{O}_D\otimes V$ is semistable
			of slope $\mu(V)$.
	\end{enumerate}
\end{lemma}

\begin{proof}
	(i) This short argument can already be found in the proof of \cite[Proposition 5.1]{bp} for line bundles of degree $1$.
        
        Let $W$ be a semistable bundle of slope $\mu$ and rank $r$ on $D$.
	The pushforward $\pi_{\ast}W$ has slope $\mu/\text{deg}(\pi)$.
	If $\pi_{\ast}W$ was not semistable, consider the maximal destabilizing
	subbundle $V$ of $\pi_{\ast}W$. By adjunction $\pi^{\ast}V\to W$ is a
	non-zero morphism of semistable bundles. As
	\[
	\mu(\pi^{\ast}V)
	=\text{deg}(\pi)\mu(V)>\text{deg}(\pi)\mu(\pi_{\ast}W)
	=\mu(W)
	\]
	this is a contradiction.
	
	(ii) Let $V$ be a semistable bundle on $C$.
	As $\pi$ is \'etale the bundle $\pi_{\ast}\mathcal{O}_D$ is
	of degree $0$ by Riemann-Hurwitz. We obtain
	\[
            \mu(V)=\mu(\pi_{\ast}(\mathcal{O}_D))+\mu(V)=\mu(\pi_{\ast}(\mathcal{O}_D)\otimes V).
        \]
	By the projection formula we have $\pi_{\ast}\mathcal{O}_D\otimes
	V\cong \pi_{\ast}\pi^{\ast}V$ which is semistable by (i) and
	Lemma \ref{lemma-stability-pullback} (iii).
\end{proof}

    As semistable bundles stay semistable under pullback by a cover $\pi:D\to C$, 
    we obtain a morphism $\pi^{\ast}:M^{ss,r,d}_C \to M^{ss,r,\deg(\pi)d}_D$.
	The finiteness of $\pi^{\ast}$ can be proven using the degree of the theta divisor.
	This can be found in
	\cite[Theorem 4.2]{heingen} and goes back to \cite[Theorem I.4]{faltings}.

	Here we give a shorter proof only using \cite[Lemma 4.3]{bp} and basic 
        properties of finite \'etale morphisms.
\begin{theorem}	\label{theorem-finite-fibres}
	Let $\pi:D\to C$ be a cover of smooth projective curves.  
	Let $r\geq 1$ and $d\in \mathbf{Z}$.
	Then the induced morphism 
	\[
	    \pi^{\ast}:M^{ss,r,d}_C\to M^{ss,r,\text{deg}(\pi)d}_D 
	\]
	is finite. If $e$ denotes the degree of the \'etale part of $\pi$,
	then the fibre of $\pi^{\ast}$ at a stable bundle $W$ on $D$ has cardinality at most $e$.
\end{theorem}

\begin{proof}
        First observe that $\pi^{\ast}$ is a morphism of projective varieties. 
        Thus, it suffices to show that it is quasi-finite.
        Furthermore, it suffices to prove the quasi-finiteness for Galois covers
        as every cover is dominated by a Galois cover.

	As each cover factors as an \'etale cover and a
	genuinely ramified cover it suffices to show the
	theorem for these two types of morphisms separately.

        The genuinely ramified case immediately follows from \cite[Lemma 4.3]{bp}.
        In fact, the lemma tells us that $\pi^{\ast}$ is injective on points: 
        If two polystable bundles on $C$ become isomorphic on $D$, then they
        are already isomorphic on $C$.

	It remains to consider the case where $\pi$ is an \'etale Galois cover.
	Let $V$ be a polystable bundle on $C$. 
	Consider the polystable bundle $\pi^{\ast}V\cong \bigoplus W_i$, 
	where the $W_i$ are stable on $D$, see Lemma \ref{lemma-stability-pullback}.
	By Lemma \ref{lemma-etale-pushforward} all bundles $\pi_{\ast}W_i$ are semistable
	of slope $\mu(V)$. The projection formula implies that
	$\pi_{\ast}(\mathcal{O}_D)\otimes V\cong\pi_{\ast}\pi^{\ast}V$.
	Thus, $V\subseteq \pi_{\ast}\pi^{\ast}V$ appears in the JH-filtration of $\bigoplus \pi_{\ast}W_i$.
	As the graded object associated to the JH-filtration
	is unique, there are only finitely many choices for $V$ if we fix $\bigoplus W_i$.

	If $V_{\mid D}\cong W$ is stable on $D$, then comparing the ranks of $V$ and $\pi_{\ast}W$ 
	we find that there can be at most $\deg(\pi)$ many different such $V$.
\end{proof}

\subsection{Strata and dimension}

In this subsection we complete the proof of Theorem 2 by a dimension estimate
on the complement of the prime to $p$ stable locus.
Consider an \'etale Galois cover $D\to C$ of smooth projective curves 
with Galois group $G$.
There are two different cases depending on whether $n=1$ or $n\geq 2$ in the decomposition $V_{\mid D}\cong \bigoplus_{i=1}^n W_i^{\oplus e}$ of Lemma \ref{lemma-pullback-galois}.
If $n=1$, then there is only one isomorphism class on which $G$ acts.
This does not mean that $W_1$ descends to $C$. However, it does up to a twist by a line bundle as it is simple.
\begin{lemma}
	\label{lemma-jochen} 
	Let $\pi:D\to C$ be an \'etale Galois cover of smooth projective curves
	with Galois group $G$.
	Let $W$ be a simple bundle of rank $r$ on $D$ which is $G$-invariant.
	Then there exists a line bundle $L$ on $D$ such that
	$W\otimes L$ descends to $C$.
\end{lemma}

\begin{proof}

	Note that for a smooth algebraic group $G$ a $G$-torsor over $C$
	corresponds to an element of $\check{H}^1_{\et}(C,G)$ 
        as a smooth morphism admits \'etale locally a section.
	The same holds for $D$.
	
	We have $H^2_{\et}(C,\mathbb{G}_m)=0$, see
	\cite[\href{https://stacks.math.columbia.edu/tag/03RM}{Tag 03RM}]{sp},
	similarly for $D$.
	By the 5-term exact sequence of the \v{C}ech to cohomology spectral sequence, see \cite[Corollary 2.10, p.101]{milne},
        we obtain the vanishing of $\check H^2_{\et}$ 
        from the vanishing of $H^2_{\et}$, i.e.,
	\[
		\check H^2_{\et}(C,\mathbb{G}_m)=0=\check H^2_{\et}(D,\mathbb{G}_m).
	\]
	Consider the short exact sequence
	
	\[
		0\to \mathbb{G}_m \to \mathrm{Gl}_r\to \mathrm{PGl}_r\to 0
	\]
        of \'etale sheaves on $C_{\et}$.
	Applying the functors $\Gamma(D,-)$ and $\Gamma(C,-)$
	we obtain a commutative diagram of exact sequences of pointed sets
 \[
    \begin{tikzcd}[column sep=1.45em]
		\check H^1_{\et}(D,\mathbb{G}_m) \ar[r] & \check H^1_{\et}(D,\mathrm{Gl}_r) \ar[r] &
		\check H^1_{\et}(D,\mathrm{PGl}_r) \ar[r]& \check H^2_{\et}(D,\mathbb{G}_m)=0\\
		\check H^1_{\et}(C,\mathbb{G}_m) \ar[u]\ar[r] &\check H^1_{\et}(C,\mathrm{Gl}_r)\ar[u] \ar[r] &
		\check H^1_{\et}(C,\mathrm{PGl}_r)\ar[u] \ar[r] & \check H^2_{\et}(C,\mathbb{G}_m) = 0 \ar[u, equal].
	\end{tikzcd}
 \]
        As $\mathbb{G}_m$ lies in the center of $\mathrm{Gl}_r$
        this sequence extends to $\check H^2$ and
	exactness at $\check H^1_{\et}(\mathrm{Gl}_r)$ is stronger than usual: 
	If two $\mathrm{Gl}_r$-torsors map to the same 
        $\mathrm{PGl}_r$-torsor they differ
	by a twist of a line bundle.
        In particular, we obtain that a $\mathrm{PGl}_r$-torsor can be lifted to a $\mathrm{Gl}_r$-torsor, 
        which also can be found in \cite[Chapter III]{sga4h}.
        
	The bundle $W$ is an element in $\check H^1_{\et}(D,\mathrm{Gl}_r)$.
	By definition of $G$-invariance we have
	isomorphisms $\psi_{\sigma}:W\xrightarrow{\sim}\sigma^{\ast}W$ for all $\sigma\in G$.
	The obstruction for descent
	$\lambda_{\sigma,\tau}:=\psi_{\sigma\tau}^{-1} \circ \tau^{\ast}\psi_{\sigma}\circ
	\psi_{\tau}$ is an isomorphism of $W$.
	By assumption $W$ is simple and $\lambda_{\sigma,\tau}$ lies in $k^{\ast}$,
	i.e., considered as a $\mathrm{PGl}_r$-torsor $W$ descends to $C$, see \cite[Theorem 1.4.46]{fga}.  
	By the surjectivity $\check{H}^1_{\et}(C,\mathrm{Gl_r})\to \check{H}^1_{\et}(C,\mathrm{PGl}_r)$
	we find a vector bundle $N$ on $C$ such that
	$N_{\mid D}\cong W$ as $\mathrm{PGl}_r$-torsors.
	Thus, the vector bundles $N_{\mid D}$ and $W$ agree up to tensoring
	with a line bundle $L$ on $D$.
\end{proof}

We are now ready to estimate the dimension of the complement of the prime to $p$
stable locus. We formulate this for arbitrary \'etale Galois covers. To obtain
the desired estimate for the prime to $p$ stable locus we apply this to 
the cover $C_{r-good}$ obtained in Theorem \ref{theorem-very-large-cover}.
\begin{lemma}
	\label{lemma-high-codimension}
	Let $\pi:D\to C$ be an \'etale Galois cover 
        of a smooth projective curve $C$ of
	genus $g_C \geq 2$.
	Let $r\geq 2$ and $d\in \mathbf{Z}$.
	Denote by $Z$ the closed subset of
	$M^{s,r,d}_C$ given by stable bundles that do not remain stable
	after pullback to $D$.
	Then $Z=Z_1\sqcup Z_2$, where 
	\[
	    Z_1:=\{ V \in M^{s,r,d}_C \mid V_{\mid D}\cong W^{\oplus e}, W\in M^{s,\frac{r}{e},\frac{\deg(\pi)d}{e}}_D, e\geq 2 \} \text{  and }
	\]
	\[
	    Z_2:=\{ V \in M^{s,r,d}_C \mid V_{\mid D}\cong \bigoplus_{i=1}^n W_i^{\oplus e}, W_i\in M^{s,\frac{r}{en},\frac{\deg(\pi)d}{en}}_D, n\geq 2 \}
	\]
	are the strata induced by Lemma \ref{lemma-pullback-galois}.
	Furthermore, 
	\[
	    \dim(Z_1)\leq r_0^2(g_C-1)+1 \text{ and } \dim(Z_2)\leq r_0r(g_C-1)+1,
	\]
	where $r_0$ is the largest proper divisor of $r$, i.e., $r_0\mid r$ and $r_0\neq r$.
	
	If $\pi$ is a prime to $p$ cover and $r$ is a power of $p$, then $Z_2$ is empty.
\end{lemma}

\begin{proof}
	Clearly, $Z=Z_1\sqcup Z_2$ by Lemma \ref{lemma-pullback-galois}.
	
	We begin with the estimate for $Z_1$.
	Consider $V\in Z_1$ and $W$ stable on $D$ such that $V_{\mid D}\cong W^{\oplus e}$ 
        for some $e\geq 2$. As the Galois group acts trivially 
        on the isomorphism class of $W$ we can apply Lemma \ref{lemma-jochen}.
	Thus, there is a line
	bundle $L$ on $D$ such that $W\otimes L\cong N_{\mid D}$ for some stable vector bundle
	$N$ on $C$.  After twisting $N$ by a line bundle on $C$, we can assume that
	$0\leq \text{deg}N<r$. Note that $W\otimes L$ is stable and so is $N$ by
	Lemma \ref{lemma-stability-pullback}. 
	Fixing the degree of $N$ fixes the degree of $L$ as
	$\text{deg}W+\frac{r}{e}\text{deg}L=\text{deg}(\pi)\text{deg}(N)$.

	We have $\det(W)^{\otimes e}\cong \det(V)_{\mid D}$ which implies
	that $L^{\otimes r}$ descends to $C$. 
	As multiplication by $r$ on $\Pic_{D/k}$ is a finite morphism,
	we obtain that the dimension of all possible line bundles $L$ (with fixed degree)
	is at most $g_C$. Write $P(f)$ for the moduli space of line bundles on $D$
	of degree $(\text{deg}(\pi)f - \text{deg}(W))\cdot \frac{e}{r}$ such that
	their $r$-th power descends to $C$, where $f$ is an integer.
	
	Let $0\leq f < r$ and fix a line bundle $L'$ of degree $f$ on $C$.
        Denote the moduli space of stable bundles of rank $r/e$ and determinant $L'$
        by $M^{s,\frac{r}{e}}_{L'}$.
	Consider the morphism 
	\[
	M^{s,\frac{r}{e}}_{L'}\times_k P(f)\to
	M^{ss,r,\text{deg}(\pi)d}_D, (N,L)\mapsto N_{\mid D}^{\oplus e} \otimes L^{-1}
	\]
	and denote the image by $Z_{f,e}$.
	Observe that $Z_{f,e}$ is closed and so is the finite union $Z'=\bigcup_{f=0}^{r-1}\bigcup_{e\mid r,e\neq 1}Z_{f,e}$.  

	The above discussion shows that $\pi^{\ast}(Z_1)\subseteq Z'$. 
	By Theorem \ref{theorem-finite-fibres}, we have that $\pi^{\ast}$ is a finite morphism and obtain 
	$\text{dim}(Z_1)=\text{dim}(\pi^{\ast}(Z_1))$.
	Computing the dimension we find 
	\[
	\text{dim}(Z_1)\leq \text{max}_{e\mid r, e\neq 1}((r/e)^2-1)(g_C-1)
	+g_C = r_0^2(g_C-1)+1,
	\]
	where $r_0$ is the largest proper divisor of $r$.  
	This concludes the estimate of $\text{dim}(Z_1)$.
		
	To obtain a bound for $\text{dim}(Z_2)$ consider $V\in Z_2$.
	By Lemma \ref{lemma-decomposition-small-degree}
	there is an intermediate cover
	$D\to D'\to C$ of degree $n$ such that $V_{\mid D'}\cong V'\oplus W'$,
	where $V'$ is semistable of rank $r/n$ and $W'_{\mid D}$ is a direct sum of conjugates of $V'_{\mid D}$.
    
    Let $\Sigma$ be a subset of $G$ of cardinality $n$.
	Consider the morphism 
	\[
	    M^{ss,\frac{r}{n},d}_{D'}\to
	    M^{ss,r,\text{deg}(\pi)d}_D,
	    V'\mapsto \bigoplus_{\sigma\in \Sigma}\sigma^{\ast}V'_{\mid D}
	\]
	and denote the image by $Z_{D',\Sigma}$.
	Observe that $Z_{D',\Sigma}$ is closed as the image of a finite morphism
	and by construction $V_{\mid D}\in Z_{D',\Sigma}$ for some $\Sigma$
	and $D'$. Thus, $\pi^{\ast}Z_2$ is contained in the union of all such
	$Z_{D',\Sigma}$, where $D\to D'\to C$ is an intermediate cover
	and $\Sigma$ is a subset of $G$ of cardinality $n$.
	Up to isomorphism there are only finitely many intermediate covers
	$D\to D'\to C$ and clearly
	there are only finitely many $\Sigma$.
	Thus, we can estimate the dimension
	\[
	\dim(Z_2)=\dim(\pi^{\ast}Z_2)\leq
	\text{max}_{D',\Sigma}\dim(Z_{D',\Sigma}),
	\]
	where $D'$ and $\Sigma$ are as above.
	Applying Theorem \ref{theorem-finite-fibres} we have
	\[
	\text{dim}(Z_{D',\Sigma})\leq \frac{r^2}{n^2}(g_{D'}-1)+1.
	\]
	By Riemann-Hurwitz we obtain
	\[
	    \dim(Z_{D',\Sigma})\leq n\frac{r^2}{n^2}(g_C-1)+1=r\frac{r}{n}(g_C-1)+1
	\]
    and conclude
	\[
	    \text{dim}(Z_2)\leq \text{max}_{n\mid r, n\neq 1}r\frac{r}{n}(g_C-1)+1=rr_0(g_C-1)+1.
	\]
	
	If $G$ is prime to $p$ and $r$ is a power of $p$, then a decomposition of the form 
	$V_{\mid D}\cong \bigoplus_{i=1}^n W_i^{\oplus e}, n\geq 2$,
	can not happen. Indeed, $n\rk(W_i)e=r$ and we find that $n$ is a power of $p$ as well. 
	By Lemma \ref{lemma-decomposition-small-degree} there is an intermediate cover of $D\to C$ of degree $n$.
	However, $G$ being prime to $p$ only allows for such an intermediate cover if $n=1$.
\end{proof}
	
	As a direct consequence of the dimension estimate we obtain
	the existence of stable bundles that remain stable on a fixed (\'etale) cover.
\begin{lemma}
	\label{lemma-etale-non-empty}
	Let $\pi:D\to C$ be an \'etale cover of a smooth projective curve $C$
	of genus $g_C\geq 2$. Let $r\geq 2$ and $d$ be integers.
	Let $Z$ be the closed
	subset $Z$ of $M^{s,r,d}_{C}$ of stable bundles that are not stable 
	after pullback to $D$.
	Then $\mathrm{codim}_{M^{s,r,d}_C}(Z)\geq 2$.
	
	In particular, there are stable bundles of rank $r$ and degree $d$
	on $C$ that remain stable after pull back to $D$.
\end{lemma}

\begin{proof}
    Observe that we can replace $D\to C$ by its Galois closure.
	By Lemma \ref{lemma-high-codimension} we have 
	\[
	    \text{dim}(Z)\leq r_0 r(g_C-1)+1,
	\]
	where $r_0$ is the largest proper divisor of $r$.
    As
	\[
	    r^2(g_C-1)+1=\text{dim}(M^{s,r,d}_C),
    \]
	$g_C\geq 2$, and $r\geq 2$, we conclude
	\[
	    \mathrm{codim}(Z)\geq r(r - r_0)(g_C-1) \geq 2.
	\]
\end{proof}

For the cover $C_{r-good}$ the estimate obtained in Lemma \ref{lemma-high-codimension}
is sharp if the rank is prime to $p$. To show this we need a way to construct
stable bundles with prescribed decomposition behaviour after pullback.
This can be done for cyclic covers. We start with a descent lemma for such covers.
\begin{lemma} 
	\label{lemma-cyclic-descend}
	Let $Y\to X$ be a cyclic \'etale cover of proper varieties 
        with Galois group $G$.
	Let $V$ be a simple sheaf on $Y$.
	Then $V$ descends to $X$ iff $V$ is $G$-invariant.
\end{lemma}

\begin{proof}
	The "only if" implication is trivial.
	For the "if" implication let $\sigma$ be a generator of $G$ of order $n$.
	Fix an isomorphism $\varphi_{\sigma}:V\xrightarrow{\sim} \sigma^{\ast}V$.
	For $2\leq l < n$ define $\varphi_{\sigma^l}:V\xrightarrow{\sim} (\sigma^l)^{\ast}V$
	inductively as the composition $\sigma^{\ast}\varphi_{\sigma^{l-1}}\circ
	\varphi_{\sigma}$. Further define $\varphi_{e}=\text{id}_V$, where $e$ denotes
	the identity of $G$.
	
	Consider $\sigma^{\ast}\varphi_{\sigma^{n-1}}\circ \varphi_{\sigma}$.
	This is an automorphism of $V$. As $V$ is simple it corresponds to 
	a scalar $\lambda \in k^{\ast}$.
	Since $k$ is algebraically closed we can find an $n$-th
	root $\lambda^{1/n}$ of $\lambda$.
	The automorphisms $\psi_{\sigma^{l}}:=\lambda^{-l/n}\varphi_{\sigma^{l}}$
	define a $G$-linearization of $V$.
	Indeed, for $1\leq l,l'$ such that $l'+l<n$ we have
	\[
	    (\sigma^{l})^{\ast}\psi_{\sigma^{l'}}\circ \psi_{\sigma^{l}}=
	\lambda^{(-l-l')/n}\cdot(\sigma^{l+l'-1})^{\ast}\varphi_{\sigma}\circ \dots \circ
	\sigma^{\ast}\varphi_{\sigma}\circ \varphi_{\sigma}=	
	\psi_{\sigma^{l+l'}}
	\]
	by definition. It remains to check this property for $l+l'=n$.
	We have
	\[
	(\sigma^{l})^{\ast}\psi_{\sigma^{l'}}\circ \psi_{\sigma^{l}}=
	\lambda^{-1}\cdot(\sigma^{n-1})^{\ast}\varphi_{\sigma}\circ \dots \circ
	\sigma^{\ast}\varphi_{\sigma}\circ \varphi_{\sigma}=\lambda^{-1}\lambda=1
	\]
	by definition of $\lambda$.  
\end{proof}

We are now able to show that the estimate in Lemma \ref{lemma-high-codimension}
is sharp for $C_{r-good}$ in most cases. It suffices to find a prime to $p$ cover
where the decomposition locus has the right dimension as the decomposition locus
for $C_{r-good}$ is the largest one.
\begin{lemma}
	\label{lemma-Z_2-strata}
	Let $C$ be a smooth projective curve of genus $g_C\geq 2$.
	Let $r\geq 2$ be such that $p$ is not the smallest proper divisor of $r$
	if $\text{char}(k)=p>0$.
	Then there is an \'etale prime to $p$ cyclic cover $D\to C$ such that 
	$\dim(Z_2)=rr_0(g_C-1)+1$, 
	where $r_0$ denotes the largest proper divisor of $r$
	and $Z_2\subset M^{s,r,d}_C$ is defined as in Lemma \ref{lemma-high-codimension}.
\end{lemma}

\begin{proof}
    As $p$ is not the smallest divisor 
    of $r$ we have that $r/r_0$ is prime to $p$.  
    Let $\pi:D\to C$ be an \'etale cyclic cover of degree $r/r_0$, i.e., with Galois group $\mu_{r/r_0}$.
    Note that such covers correspond to torsion points of order $r/r_0$ in $\Pic^0_{C/k}$ and always exist.
    Also note that $r/r_0$ is prime. Thus, there are no intermediate covers.
    
    Consider
    \[
        U:=M^{s,r_0,\deg(\pi)d}_D \cap (M^{ss,r_0,\deg(\pi)d}_D \setminus \pi^{\ast} M^{ss,r_0,d}_C).
    \]
    By Theorem \ref{theorem-finite-fibres} pullback along $\pi$ is a finite morphism and the set $U$ is open and non-empty. 
    Thus, by Riemann-Hurwitz $U$ has dimension 
    \[
    \dim(U)=r_0^2(g_D-1)+1=r r_0(g_C-1)+1.
    \]
    
    Consider a closed point $W'\in U$. Then the orbit $O$ of $W'$ under the action of $\mu_{r/r_0}$ is contained in $M^{s,r_0,\deg(\pi)d}_D$.
    By Lemma \ref{lemma-cyclic-descend} the orbit $O$ has cardinality $r/r_0$ as otherwise $W'$ would descend to $C$.
    Clearly, no conjugate of $W'$ can descend to $C$ as well, i.e., $O\subset U$.
    
    Consider the bundle $W:=\bigoplus_{\sigma \in \mu_{r/r_0}}\sigma^{\ast}W'$.
    Then $W$ has rank $r$, admits a $\mu_{r/r_0}$-linearization, and no polystable summand of $W$ admits a
    $\mu_{r/r_0}$-linearization. Thus, there exists $V\in M^{s,r,d}_C$ such that $V_{\mid D}\cong W$.
    By construction $V$ lies in $Z_2$. In particular, $\pi^{\ast}Z_2$ 
    contains the image of
    \[
        U\to M^{s,r,\deg(\pi)d}_D, W'\mapsto \bigoplus_{\sigma \in \mu_{r/r_0}}\sigma^{\ast}W',
    \]
    which has dimension
    $\dim(U)$. We obtain that 
    \[
        \dim(Z_2)\geq r r_0(g_C-1)+1.
    \]
    As we already have the other inequality from Lemma
    \ref{lemma-high-codimension} we conclude.
\end{proof}

\begin{remark}
    Let $D\to C$ be an \'etale Galois cover.
    One can further decompose $Z_2$ into the strata $Z_2(n,e):=\{ V\in M^{s,r,d}_C | V_{\mid D} \cong \bigoplus_{i=1}^n W_i^{\oplus e} \},$
    where $V_{\mid D}\cong \bigoplus_{i=1}^n W_i^{\oplus e}$ is the decomposition of Lemma
    \ref{lemma-pullback-galois}. One can compute the dimension of $Z_2(n,1)$ in an analogous manner
    if $D\to C$ is prime to $p$ and cyclic of degree $n$. The only change being that one has to remove
    all bundles of rank $r/n$ arising from an intermediate cover $D\to D'\to C$, $D'\neq D$.
\end{remark}

Applying the results of this subsection for arbitrary Galois covers to the cover $C_{r-good}\to C$, 
see Theorem \ref{theorem-very-large-cover}, we obtain Theorem $2$:

\begin{theorem}
	\label{theorem-non-empty}
	Let $C$ be a smooth projective curve of genus $g_C\geq 2$.
	Let $r\geq 2$ and $d\in\mathbf{Z}$. 
        Then the prime to $p$ stable bundles of rank $r$
	and degree $d$ form a big open $M^{p'-s,r,d}_C$ in $M^{s,r,d}_C$.
	More precisely, we have
	\[
		\dim(M^{s,r,d}_C\setminus M^{p'-s,r,d}_C)\leq rr_0(g_C-1)+1,
        \]
	where $r_0$ denotes the largest proper divisor of $r$.
        Moreover, if $p$ is not the smallest proper divisor of $r$, then equality holds.
\end{theorem}

Extending a prime to $p$ stable vector bundle from a large curve to a surrounding smooth projective variety using Mathur's extension theorem, \cite{mathur} Theorem 1, we obtain the existence of prime to $p$ stable vector bundles in higher dimensions. 
However, we can not control the numerical data, i.e., which components of the stack of bundles 
admit prime to $p$ stable bundles.
\begin{cor}
    Let $X$ be a smooth projective variety of dimension $\geq 2$.
    There are prime to $p$ stable vector bundles of rank $r\geq \dim(X)$ on $X$.  
\end{cor}

As the general bundle is prime to $p$ stable, we obtain:
\begin{cor}
\label{cor-not-dense}
    Let $C$ be a smooth projective curve of genus $g_C\geq 2$.
    Let $r\geq 2$. Then the stable bundles of rank $r$
    that are trivialized on a prime to $p$ cover
    are not dense in $M^{s,r,0}_C$.
\end{cor}

\section*{Acknowledgements}
The results presented are part of the PhD thesis of the author carried out at the University of Duisburg-Essen.
The author would like to thank his PhD advisor Georg Hein for his support and many inspiring and fruitful discussions.
He would also like to thank his second advisor Jochen Heinloth for some helpful discussions, in particular Lemma \ref{lemma-jochen}.
Thanks also go to Stefan Reppen for pointing out the literature \cite{dm} and a discussion on \'etale trivializable bundles. The author would also like to thank the anonymous referees for their feedback which improved and clarified the paper.
The author was funded by the DFG Graduiertenkolleg 2553.

\bibliographystyle{plain}

\bibliography{bibliography}
\end{document}